\documentclass[leqno,11pt]{article}

\usepackage{amsmath}
\usepackage{amsthm}
\usepackage{amssymb}
\usepackage{amscd}
\usepackage{amsfonts}
\usepackage{amsbsy}

\newtheorem{proposition}{Proposition}
\newtheorem{lemma}{Lemma}
\newtheorem{corollary}{Corollary}
\newtheorem{definition}{Definition}
\newtheorem{remark}{Remark}
\newtheorem{example}{Example}

\begin{document}

\title{Side conditions for ordinary\\
differential equations}

\author{G. Cicogna\\
      Dipartimento di Fisica, Universit\`a di Pisa\\
        and INFN, Sezione di Pisa\\
       Largo B.~Pontecorvo 3\\
     I-56127 Pisa, Italy\\
       \\
G. Gaeta\footnote{Research partially supported by MIUR-PRIN program under project 2010--JJ4KPA}\\
 Dipartimento di Matematica, Universit\`a degli Studi di Milano\\
via Saldini 50\\
I-20133 Milano, Italy\\
             \\
S. Walcher\\
Lehrstuhl A f\"ur Mathematik, RWTH Aachen\\
D-52056 Aachen, Germany
}

%\date{}

\maketitle

\begin{abstract}
We specialize Olver's and Rosenau's side condition heuristics for the determination of particular invariant sets of ordinary differential equations. It turns out that side conditions of so-called LaSalle type are of special interest. Moreover we put side condition properties of symmetric and partially symmetric equations in a wider context. In the final section we present an application to parameter-dependent systems, in particular to quasi-steady state for chemical reactions.
\end{abstract}

\section{Introduction and overview}

Systems of ordinary or partial differential equations which admit a local Lie transformation group of symmetries (equivalently, infinitesimal symmetries) have been the object of intense research activity in the past decades. As representatives for numerous monographs on this subject we mention only the classical works by Bluman and Cole \cite{BluCo}, and Olver \cite{Olv}.
Symmetric systems admit two characteristic features. First, one may reduce the system
(locally, near points of maximal orbit dimension) to an orbit space with respect to the group action; this is realized via group invariants. 
Second, from symmetries one obtains special invariant sets, such as group-invariant solutions (also known as relative equilibria in the ODE case).
\\
However, it has been noticed that the feature of reducibility also holds for systems that are not necessarily symmetric. Olver and Rosenau \cite{OlRo2} discussed this phenomenon in detail for systems of partial differential equations. For ordinary differential equations, a reduction method which is based on a generalization of lambda symmetries (Muriel and Romero \cite{MuRo}) was introduced and analyzed in \cite{CGW1, CGW2} from different perspectives.\\
Likewise, it has been observed that a differential equation, even if not admitting a given group as symmetry group, may very well have particular solutions which are invariant with respect to such a group. Among a large number of relevant contributions, we mention the notion of {\em conditional symmetry} due to Bluman and Cole \cite{BluCo} (see also  Fushchich and Tsyfra \cite{FuTs}, and Levi and Winternitz \cite{LW1, LW2} who outlined an algorithmic approach), and the related notion of {\em weak symmetry} introduced by Pucci and Saccomandi \cite{PS2}. These in turn were generalized to the notion of {\em partial Lie point symmetry} in \cite{CGPart}.
Olver and Rosenau \cite{OlRo1} set all these observations in a general framework (mostly for partial differential equations) by noting that the additional conditions which determine possible particular solutions need not originate from a group action, and that such {\em side conditions} may a priori be chosen quite freely. \\
Herein lies the motivation for the present note: We will discuss the side condition approach as set down in \cite{OlRo1} for the class of ordinary differential equations; moreover we will identify some settings and present some applications for which this heuristics seems promising.\\

The general framework is as follows. Let a  first order autonomous ODE
\begin{equation}\label{ode} dx/dt\ = \ f(x) \end{equation} 
be given on an open subset $U$ of
${\mathbb K}^n$, with $\mathbb K$ standing for $\mathbb R$ or $\mathbb C$. 
All functions and vector fields are required to be smooth.  The vector field associated to \eqref{ode} will be denoted by $X_f$ (thus $X_f(\psi)(x)=D\psi(x)f(x)$ for any scalar-valued function), and the local flow of \eqref{ode}, i.e. the solution of the initial value problem for $y$ at $t=0$, will be called $F(t,\,y)$.
For some results  we will require analyticity of vector fields and functions, and even further restrict to the local analytic case. In addition we will discuss some special results for polynomial functions and vector fields. Non-autonomous equations are, as usual, identified with orbit classes of autonomous systems (see e.g. the remarks in \cite{CGW1}).
Restricting attention to open subsets of $\mathbb K^n$ (rather than
manifolds) imposes no loss of generality for local considerations.\\

The plan of the paper is as follows. We first introduce and discuss the pertinent notions (in particular the notion of {\em algebraic side condition}), derive necessary and sufficient criteria for admissibility of side conditions, and give examples to illustrate the concepts. As should be expected, these are rather straightforward and transparent for ordinary differential equations. The criteria are useful in trial-and-error approaches, but they do not lead to an algorithm for determining nontrivial side conditions. Thus further a priori restrictions or specializations, possibly motivated by properties of a model underlying the equation, are necessary to allow for stronger results.\\
We discuss two such restrictions. The first is motivated by a theorem of LaSalle on limit sets of dynamical systems, therefore we will speak of {\em LaSalle type} side conditions; it is possible to characterize the invariant sets these define in a more concise manner. Second, we review and generalize the classical side conditions induced by local transformation groups (not necessarily symmetry groups of \eqref{ode}) and also include partial local symmetries. As an application we discuss side conditions for two-dimensional systems.\\
 In the final section we consider side conditions for parameter-dependent systems. There are some practically relevant settings where desirable properties of a differential equation (motivated e.g. by experiments or intuition regarding the underlying  physical system) naturally lead to side conditions. A case in point is quasi-steady state (QSS) for chemical reaction equations. Our discussion of some well-known equations illustrates that the side condition approach provides a systematic and transparent way to identify appropriate ``small parameters" for QSS.
\section{Side conditions}
\subsection{Basics}
For starters we recall an invariance criterion; see e.g. \cite{GSW}, Lemma 3.1:
\begin{remark} \label{invcrit}{\em The common zero set of smooth functions $\psi_1,\ldots,\psi_r$ on $U$ is invariant for \eqref{ode} if there exist smooth functions $\nu_{jk}$ on $U$ such that
\[
X_f(\psi_j)=\sum_k\nu_{jk}\psi_k,\quad 1\leq j\leq r.
\]
\rightline {$\diamond$}}
\end{remark}
The basic specialization of the side condition approach to ordinary differential equations is as follows.
\begin{definition}\label{algsidecon} Let smooth functions
$\gamma_1 , \, \ldots \, , \gamma_s : \, U \to \mathbb K$
be given, with common zero set  $W$. We say that equation \eqref{ode} {\em admits the algebraic side conditions}  $\gamma_1=\cdots=\gamma_s=0$ if there exists a solution trajectory of \eqref{ode} which is contained in $W$.
\end{definition}
In other words, 
we look for solutions of  \eqref{ode} that are restricted to some prescribed
``subvariety" (i.e., a subset defined by finitely many smooth equations). The definition implies that $W$ is nonempty, but we do not require the whole subset to be invariant for the system.

\begin{proposition}\label{algsi}  Let smooth functions be given as in Definition \ref{algsidecon}.\\
{\em(a)} The differential equation \eqref{ode} admits the algebraic side conditions $\gamma_1=\cdots=\gamma_s=0$ only if the set of functions
$$ \left\{ X_f^k (\gamma_j) ; \, k \geq 0 , \, 1 \leq j \leq s \right\} = \\
\left\{ \gamma_1, \ldots, \gamma_s, \, X_f(\gamma_1) , \ldots ,
X_f(\gamma_s) \, , X_f^2(\gamma_1) , \ldots \right\} $$ has a
common zero. \\
{\em(b)} If $f$ and all $\gamma_i$ are analytic (or polynomial), the converse holds: Whenever the common zero set  $\widetilde{W}$ of
the $ X_f^k (\gamma_j)$ is not empty then it is invariant for \eqref{ode}.
\end{proposition}
\begin{proof} Consider the local flow $F$. The first assertion follows from the relation
$$ \frac{d^k}{dt^k}\gamma_j(F(t,\,y))=X_f^k(\gamma_j)(F(t,\,y)).$$
The second assertion is a consequence of the Lie series formula (see e.g. Groebner and Knapp \cite{GrKn})
$$ \gamma_j(F(t,y))=\sum_{k\geq 0}\frac{t^k}{k!}X_f^k(\gamma_j)(y). $$
\end{proof}
\begin{remark}{\em  In the local
analytic setting, finitely many of the $X_f^k(\gamma_j)$ suffice to
determine $\widetilde{W}$, and the same statement applies to polynomial vector fields and functions $\gamma_i$. In other words, the criterion from Remark \ref{invcrit} will hold for a generating set consisting of finitely many  $X_f^k(\gamma_j)$. This is due to the Noetherian property of the power series and polynomial rings in $n$ variables (see e.g. Ruiz \cite{Rui} and Kunz \cite{Kun}). \\
\rightline {$\diamond$}}
\end{remark}
\begin{remark}\label{nonrob}
{\em The property of admitting a given side condition is not robust with respect to small perturbations, as will be seen in the examples below. The more appropriate question seems whether a perturbation of such a (smooth or analytic) system will admit a suitably perturbed side condition. One result concerning this problem is given in Proposition \ref{asnear} below.}
\rightline {$\diamond$}
\end{remark}
At first sight, transferring Olver's and Rosenau's approach from \cite{OlRo1}
 to the setting of ordinary differential
equations should involve more than just algebraic side
conditions. Rather it may seem appropriate to consider ``(ordinary) differential side
conditions'', i.e., to assume that \eqref{ode} is augmented by additional ordinary
differential equations which lead to an overdetermined system.
But as the following elementary observation shows, the existence of such differential side conditions is equivalent to the existence of
algebraic side conditions.
\begin{remark}\label{diffsi} {\em (a) Let $\phi_1 , \ldots , \phi_s$
and $\rho_1 , \ldots , \rho_s$ be smooth functions on $U$. Assume that some solution $z(t)$ of \eqref{ode} satisfies additional differential conditions of first order, of the type
\begin{equation}\label{diffside}
\frac{d}{dt}\phi_j(z(t)) = \rho_j(z(t))
\ , \ \ 1 \leq j \leq s \ . 
\end{equation}
Then $z(t)$ is contained in the common zero set of the functions
\[
\theta_j := X_f(\phi_j) - \rho_j,\quad1\leq j\leq s.
\]
Conversely, if the common zero set of these $\theta_j$ contains a nonempty invariant set of \eqref{ode}, then there exists a solution of \eqref{ode} which satisfies the differential side conditions \eqref{diffside}.
 This is a direct consequence of the relation
$$ \frac{d}{dt}\phi_j(z(t)) = X_f(\phi_j)(z(t)) $$
for any solution $z(t)$ of \eqref{ode}.\\
(b) Higher order ordinary differential side conditions may be rewritten as systems of first order side conditions, as usual, hence as algebraic side conditions.\\
(c) More generally one could think of ``partial differential side conditions", thus regarding \eqref{ode} as a degenerate partial differential equation system for which only one independent variable $t$ occurs explicitly. But by this ansatz one would return to the general framework of \cite{OlRo1}; we will not pursue it further.}\\
\rightline {$\diamond$}
\end{remark}
\subsection{Examples }

We provide a few examples for algebraic and differential side conditions, to illustrate the procedure, and to show that the heuristics will provide nontrivial information only in special circumstances. Examples \ref{expara2} and \ref{excirc} involve differential side conditions.

\begin{example}\label{expara1}{\em 

Consider \eqref{ode} with
$$
f(x)=\left(\begin{array}{c} x_1-x_2^2+x_3 \\ x_3\\
x_1+x_1^2+2x_2x_3\end{array}\right)\ .
$$
To determine invariant sets contained in the zero set of
$\gamma(x):=x_1$, compute $X_f(\gamma)\,(x)= x_1-x_2^2+x_3$ and
furthermore
$$
X_f^2(\gamma)\,(x)= (2+x_1)x_1-x_2^2+x_3=  (1+x_1)\gamma(x)+ X_f(\gamma)(x).
$$
The last equality implies that the common zero set of all the $X_f^k(\gamma)$, which is invariant according to Proposition \ref{algsi}, is equal to the common zero set of $\gamma$ and
$X_f(\gamma)$. Thus the parabola, defined by $x_1=x_2^2-x_3=0$,
is invariant for $f$.}
\end{example}

\begin{example}\label{exlasa}{\em 

In the qualitative theory of ordinary differential equations the setting of Proposition \ref{algsi} occurs
naturally: Assume that \eqref{ode} admits a Lyapunov function
$\psi$ on $U$. Then the LaSalle principle (see e.g. Hirsch, Smale and Devaney \cite{HSD}, Ch.~9) states that
any omega-limit set is contained in the zero set of
$\gamma:=X_f(\psi)$, thus all nonempty limit sets are obtained from the side condition $X_f(\psi)$. \\
As a specific example, consider the motion in an $n$-dimensional potential $\psi$ with generalized linear friction; i.e. the  system
\[
\begin{array}{rcl}
\dot x_1&=& x_2\\
\dot x_2&=& -{\rm grad}\,\psi(x_1) -A\,x_2
\end{array}
\]
in $\mathbb R^{2n}$, with a matrix $A$ whose symmetric part $A+A^{\rm tr}$ is positive definite. For $\theta(x)= \left<x_2,\,x_2\right> +2\psi(x_1)$ (with the brackets denoting the standard scalar product in $\mathbb R^n$) one finds $X_f(\theta)=-\left<(A+A^{\rm tr})x_2,\,x_2\right>$, whence $\theta$ is a Lyapunov function and any limit set is contained in the zero set of (all components of) $x_2$. By invariance of limit sets one arrives at the familiar conclusion that any limit point of the system is stationary.
}
\end{example}

\begin{example}\label{expara2}{\em 

Consider again \eqref{ode} with $f$ from Example \ref{expara1}.
Assume that this equation admits the differential side condition
$\frac{d}{dt}(x_1^2+x_2) = x_3$ . Using Remark \ref{diffsi},  $\phi:=x_1^2+x_2$ and
$\rho:=x_3$, one sees that $z(t)$ lies in the zero set of
$$ \gamma:= X_f(\phi) - \rho = 2 x_1^2 - 2 x_1 x_2^2 + 2 x_1 x_3 - x_3 \ , $$
and proceeding with straightforward
computations (that are omitted here) provides only the obvious invariant set
$\{0\}$. The heuristics yields no interesting information here. See, however, the following example.}
\end{example}
\begin{example}\label{excirc}{\em 
Given  the two-dimensional system with
$$
f(x)=\left(\begin{array}{c} -x_1-x_2+x_1^3+x_1x_2^2\\
(1+\beta)x_1+x_2-x_1^3-x_1^2x_2-x_1x_2^2-x_2^3\end{array}\right),
\quad \beta\in {\mathbb R} $$ we search for a solution $z(t)$ such
that $\ddot z_1=-z_1$. We transfer this side condition to first order
by setting
$\phi_1:=x_1,\,\phi_2:=X_f(x_1)=-x_1-x_2+x_1^3+x_1x_2^2$,
hence we search for solutions contained in
the common zero set of
\[
\begin{array}{rcl}
\gamma_1:&=&X_f(\phi_1)-\phi_2\\
\gamma_2:&=&X_f(\phi_2)+\phi_1

\end{array}
\]
(the differential side condition having been transformed to the algebraic side condition $\gamma_2=0$ via Remark \ref{diffsi}).
Consider first the case $\beta=1$. Setting
$\sigma:=x_1^2+x_2^2-1$, a straightforward calculation shows that
the circle given by $\sigma=0$ is a possible candidate (since
$\sigma$ is a factor in $\gamma_2$), and furthermore that
$z_1(t)=\cos t,\,z_2(t)=\sin t$ provides indeed a solution with
the desired properties. For $\beta\not=1$, an equally
straightforward (but more tedious) calculation shows that the
$X_f^k(\gamma_j)$ have only $0$ as common zero, thus for $\beta\not=1$ the only
solution satisfying the differential side condition is trivial.}
\end{example}
\begin{example}\label{exnonrob} {\em This example illustrates Remark \ref{nonrob}. Consider the system 
\[
\begin{array}{rcl}
\dot x_1&=& x_1\cdot \phi(x_1,\,x_2) + \varepsilon \nu(x_2)\\
\dot x_2&=& \psi(x_1,\,x_2)

\end{array}
\]
with smooth functions $\phi$, $\nu$, $\psi$ of the indicated variables, and a parameter $\varepsilon\in\mathbb R$. For $\varepsilon=0$ the zero set of $x_1$ is invariant for the system, but for $\varepsilon \not=0$ there exists an invariant set admitting the side condition $x_1=0$ if and only if $\nu$ and the function defined by $x_2\mapsto \psi(0,\,x_2)$ have a common zero.  The less restrictive question about the existence of an invariant set given by an equation $x_1+\varepsilon\rho(x,\varepsilon)=0$ certainly has an affirmative answer (locally, e.g. near $0$) whenever $\phi(0,\,0)\not=0$ and $\psi(0,\,0)\not=0$; see also Proposition \ref{asnear} below.
}
\end{example}
\subsection{Side conditions of LaSalle type}
As the examples indicate, a trial-and-error side condition ansatz will generally not yield any invariant sets, thus for a prescribed side condition $\gamma$ the common zero set of the $X_f^k(\gamma)$ will generally be empty. To make the general side condition heuristics more workable, appropriate restrictions should be imposed. Example \ref{exlasa} may serve as a motivation for one such restriction; note that at least the stationary points of $f$ satisfy any side condition of this type. Thus we are led to:
\begin{definition} We call a side condition $\gamma$ to be of {\em LaSalle type} if there is some function $\theta$ such that $\gamma=X_f(\theta)$.
\end{definition}
Side conditions of LaSalle type generalize a scenario which occurs in a different -- and quite familiar -- context.
\begin{remark} {\em Let $\phi$ be a function of $n$ variables. If the first order system \eqref{ode} has the form
\[
f(x)=\left(\begin{array}{c}x_2\\
                                    \vdots\\
                                     x_n\\
                              \phi(x_1,\ldots,x_n)\end{array}\right),
\]
thus corresponds to the $n^{\rm th}$ order differential equation
\[
x^{(n)}= \phi(x,\,\dot x,\,\ldots, x^{(n-1)})
\]
then it admits the side condition $\gamma=X_f(x_1)$ of LaSalle type. The invariant set obtained from this side condition is precisely the set of stationary points.}\\
\rightline {$\diamond$}
\end{remark}
Invariant sets obtained from LaSalle type side conditions can be characterized more precisely and, to some extent, obtained in an algorithmic manner. In order for the hypotheses of the following Proposition to be satisfied, one may have to pass from $U$ to a suitable open subset (which is dense in the analytic setting for connected $U$).
\begin{proposition}\label{lasaprop} Let $\theta:\,U\to \mathbb K$, and let $k$ be a positive integer such that $\left\{\theta,\,X_f(\theta),\ldots,X_f^k(\theta)\right\}$ are functionally independent but 
\[
\left\{\theta,\,X_f(\theta),\ldots,X_f^k(\theta),X_f^{k+1}(\theta)\right\}
\]
are functionally dependent at all points of $U$. By the implicit function theorem there exists a function $\mu$ of $k+2$ variables such that
\begin{equation}\label{vanish}
\mu(\theta,\,X_f(\theta),\ldots,X_f^k(\theta),X_f^{k+1}(\theta))=0\quad{\rm on}\,\, U.
\end{equation}
Denoting by $D_{i}$ the partial derivative with respect to the $i^{\rm th}$ variable, define
\[
U^*:=\left\{x\in U;\, D_{k+2}\,\mu\left(\theta(x),0,\ldots, 0\right)\not=0\right\}.
\]
Then the subset 
\[
Z:=\left\{ x\in U^*;\,X_f(\theta)(x)=\cdots=X_f^{k+1}(\theta)(x)=0 \right\}
\]
 is invariant for the restriction of \eqref{ode} to $U^*$.
\end{proposition}
\begin{proof} Taking the Lie derivative of the identity \eqref{vanish} one obtains
\[
0=\sum_{i=1}^{k+2} D_i\mu\left(\theta,\,X_f(\theta),\ldots,X_f^{k+1}(\theta)\right)\cdot X_f ^iî(\theta)
\]
and invariance follows from Remark \ref{invcrit}.
\end{proof}
\begin{remark}\label{lasarem} {\em (a) The Proposition suggests the following trial-and-error approach: For a ``test function" $\theta$ form $X_f(\theta),\,X_f^2(\theta)\ldots$ and stop at the smallest $\ell$ such that the functions
\[
\theta,\,X_f(\theta),\ldots, X_f^{\ell+1}(\theta)
\]
are functionally dependent on $U$. Then check the common zero set of $X_f(\theta),\ldots, X_f^{\ell+1}(\theta)$ for invariance. Here, invariant sets of positive dimension are of particular interest.
\\
(b) For polynomial functions and vector fields it is known that $\mu$ may be chosen as a polynomial (which is algorithmically accessible), and relation \eqref{vanish} will hold throughout $\mathbb K^n$.
}\\
\rightline {$\diamond$}
\end{remark}
\begin{example}\label{exlasa3d}{\em Let
\[
f(x)=\left(\begin{array}{c}
x_2+x_3-x_1x_2-x_2x_3-x_2^2x_3\\
x_2-x_2^2+x_2x_3\\
x_1+x_3+x_2^2
\end{array}\right)
\]
and
\[
\theta(x)=x_1+x_2x_3\,.
\]
One computes
\[
\begin{array}{rcl}
X_f(\theta)&=& x_2+x_3\\
X_f^2(\theta)&=& x_1+x_2+x_3+x_2x_3
\end{array}
\]
which shows that $X_f^2(\theta)=X_f(\theta)+\theta$, and by Proposition \ref{lasaprop} the common zero set $Z$ of $X_f(\theta)$ and $X_f^2(\theta)$ (which is a parabola in $\mathbb R^3$, defined by $x_1=x_3^2$ and $x_2=-x_3$) is invariant for the system. 
}
\end{example}
At this point, a few words on the practicality of the approach may be in order. Example \ref{exlasa3d} was actually tailored to illustrate a nontrivial application of Proposition \ref{lasaprop} (i.e., yielding an invariant set of positive dimension), but it should be noted that the trial-and-error approach can indeed be systematized for polynomial equations, using standard methods from algorithmic algebra (for these see e.g. Cox, Little and O'Shea \cite{CLO}). Given a polynomial vector field $f$ on $\mathbb K^n$, one may start with a polynomial ``test function" of fixed degree, with undetermined coefficients (e.g. $\theta$ of degree one, with 
$n$ undetemined coefficients of $x_1,\ldots,x_n$) and evaluate the determinantal condition for functional dependence of
\[
\theta,\,X_f(\theta),\ldots, X_f^{k-1}(\theta),\quad k\leq n.
\]
This in turn will provide conditions on the undetermined coefficients in the test function. If a nontrivial test function remains, proceed to determine a polynomial $\mu$ as in Remark \ref{lasarem} (see \cite{CLO}, Ch.~3 for this step) and apply the Proposition. In this way one has an algorithmic approach to determine invariant sets, which will indeeed work for the above example (starting with undetermined test functions of degree 2). But, since polynomial vector fields generally do not possess algebraic invariant sets of positive dimension, the search may still yield only trivial results.\\
For a variant see also the final section if this paper. Similar observations apply, in principle, to the local analytic case.

\section{Side conditions which inherit properties of\\ symmetric systems}
We return to our vantage point of imposing (appropriate) restrictions on the class of side conditions to be investigated. Historically, the concept seems to have emanated from the group analysis of differential equations.
Thus, side conditions were (and are) constructed from local transformation groups and equations which determine group orbits; see Bluman and Cole \cite{BluCo}, Levi and Winternitz \cite{LW1, LW2} and Pucci and Saccomandi \cite{PS2}, to name only a few references.
In this section we follow the classical approach by first recalling typical invariant sets of symmetric systems \eqref{ode}, which we then take as a motivation for particular types of side conditions. In dimension two there is a relatively strong correspondence between side conditions and symmetry properties.
\subsection{Invariant sets from symmetries}
We first assume that \eqref{ode} admits smooth orbital infinitesimal symmetries $g_1,\ldots, g_r$ on $U$; hence there exist smooth functions $\alpha_i$ on $U$ such that the identities
\begin{equation}\label{sym}
\left[ g_i,\,f\right]=\alpha_i f,\,1\leq i\leq r
\end{equation}
hold for the Lie brackets 
\[
\left[g_i,\,f\right](x)=Df(x)\,g_i(x)-Dg_i(x)\,f(x)
\] 
throughout $U$. Let us recall some basic results on group-invariant solutions and generalizations. 

\begin{proposition}\label{groupinv} {\em (a)} If \eqref{sym} holds then the set
\[
Y:=\left\{ x\in U;\,\dim_{\mathbb K^n}\left<f(x),\,g_1(x),\ldots,g_r(x)\right>\leq r\right\}
\]
(the brackets denoting the linear subspace spanned by a set of vectors here and in the following) is invariant for \eqref{ode}.\\
{\em (b)} If all $\left[ g_i,\,f\right]=0$ then 
\[
Z:=\left\{ x\in U;\,\dim_{\mathbb K^n}\left<g_1(x),\ldots,g_r(x)\right>\leq r-1\right\}
\]
is invariant for \eqref{ode}.
\end{proposition}
See for instance \cite{WMul}, Theorem 3.1. Note that no assumptions were made about any relation among the $g_i$. 
\begin{remark}{\em  There exist different characterizations of the sets above.\\

\noindent(a) One has $x\in Y$ if and only if 
$\widetilde\Delta\left(f(x),\,g_1(x),\ldots,g_r(x)\right)=0$
for every alternating $(r+1)$-form $\widetilde\Delta$.\\
(b) One has $x\in Z$ if and only if 
$\Delta\left(g_1(x),\ldots,g_r(x)\right)=0$
for every alternating $r$-form $\Delta$.}\\
\rightline {$\diamond$}

\end{remark}

\begin{remark} {\em (a) If $r=1$ then the infinitesimal symmetry $g_1$ generates a local one-parameter group, and $Y$ is the union of group-invariant solutions (in the sense of \cite{Olv}, Section 3.1) and stationary points of $g_1$. For arbitrary $r$, if the $g_i$ span a finite dimensional Lie algebra, one obtains the group-invariant solutions by taking the intersection of all the sets defined by $\widetilde \Delta(f,\,g_i)=0$, with every alternating 2-form $\widetilde \Delta$.\\
(b) In some settings, Proposition \ref{groupinv} provides all relevant invariant sets. For instance, if the $g_i$ span the Lie algebra of a reductive linear algebraic group then all common invariant sets of the differential equations admitting the infinitesimal symmetries $g_1,\ldots, g_r$ can be obtained from $\Delta(g_1^*(x),\ldots,g_s^*(x))=0$, with suitable linear combinations $g_j^*$ of the $g_i$, and $\Delta$ running through all alternating $s$--forms, $s\leq r$, and set-theoretic operations. See \cite{GSW}, Theorem 3.6.\\
(c) If \eqref{sym} holds and some $\alpha_i\not=0$ then $Z$ is not necessarily invariant for \eqref{ode}. A simple example in $\mathbb K^2$ is 
\[
f(x)=\left(\begin{array}{c}1\\ 0 \end{array}\right),\quad g(x)=\left(\begin{array}{c}x_1\\ 0 \end{array}\right) \mbox{  with  } \left[g,\, f\right]=-f. 
\]
The set of all $x$ with $g(x)=0$ (in other words, $x_1=0$) is clearly not invariant for \eqref{ode}.
}\\
\rightline {$\diamond$}
\end{remark}
From a suitable relaxation of condition \eqref{sym} one still obtains invariant sets of \eqref{ode}. Assume that there are smooth functions $\alpha_i$, $\sigma_{ij}$ on $U$ such that 
\begin{equation}\label{orbred}
\left[ g_i,\,f\right]=\alpha_i f +\sum_j \sigma_{ij}\,g_j,\,1\leq i\leq r.
\end{equation}
If the $g_i$ are in involution then this condition characterizes local orbital reducibility of \eqref{ode} by the common invariants of $(g_1,\ldots,g_r)$; see \cite{CGW1, CGW2}. Moreover, if all the $\alpha_i=0$ then one has local reducibility. (If the $g_i$ span a finite dimensional Lie algebra then we have reduction of non-symmetric systems by group invariants; cf. Olver and Rosenau \cite{OlRo2}, as well as \cite{CGW1}.) But the following statements hold true even when the $g_i$ do not form an involution system.
\begin{proposition}\label{sigmainv} {\em (a)} Assume that \eqref{orbred} holds on $U$. Then the set $Y$, as defined in Proposition \ref{groupinv}, is invariant for \eqref{ode}.\\
{\em(b)} If, in addition, all $\alpha_i=0$ then the set $Z$, as defined in Proposition \ref{groupinv}, is invariant for \eqref{ode}.
\end{proposition}
For a proof see \cite{CGW1}, Corollary 2.9 and Theorem 2.19, with a slight modification of some arguments. Following the approach in Bluman and Cole \cite{BluCo}, Levi and Winternitz \cite{LW1, LW2}, Pucci and Saccomandi \cite{PS2}, among others, one will consider the sets defined by  Proposition \ref{groupinv} ff. as candidates for side conditions. \\
It may be appropriate to illustrate the various concepts and their interrelation, thus we give a small example. One may generalize the underlying construction and the arguments to connected compact linear groups and their Lie algebras; see \cite{CGW1}, Lemma 2.25.
\begin{example}\label{exgroup}{\em 
Let $\alpha$ and $\beta$ be smooth on $\mathbb R^2\setminus\left\{ 0\right\}$, and 
\[
f(x)=\alpha(x)\,\left(\begin{array}{c} x_1\\
                                                       x_2\end{array}\right)+\beta(x)\,\left(\begin{array}{c} -x_2\\
                                                       x_1\end{array}\right),\quad g(x)= \left(\begin{array}{c} -x_2\\
                                                       x_1\end{array}\right).
\]
(Note that every smooth vector field $f$ in $\mathbb R^2$ admits such a representation on $\mathbb R^2\setminus\{0\}$.)
Now $g$ is an infinitesimal symmetry of $f$ (in other words, the differential equation is $SO(2)$-symmetric) if and only if both $\alpha$ and $\beta$ are functions of $\phi(x)=x_1^2+x_2^2$ only. The differential equation \eqref{ode} is reducible by the invariant $\phi$ of $SO(2)$ if and only if $\alpha$ is a function of $\phi$ only; see \cite{CGW1}, Proposition 2.26. More generally, motivated by Proposition \ref{groupinv}, one may consider the side condition
\[
\gamma(x):=\det(f(x),\,g(x))=\alpha(x)\cdot\phi(x),
\]
thus investigate the zero set $Z$ of $\alpha$ for invariant subsets of \eqref{ode} in $\mathbb R^2\setminus\{0\}$. Any nonstationary invariant subset $\widetilde Z$ of $Z$ contains an arc of a circle $\phi(x)={\rm const.}\not=0$, since $\beta(x)\not=0$ for $x\in\widetilde Z$, hence the trajectory must be locally invariant for $g$. In the analytic setting, this is equivalent to invariance of the whole circle. Thus via the side condition $\gamma$ one will obtain stationary points and invariant circles centered at the origin. For a system admitting an invariant circle, assuming some genericity conditions, one finds via the Poincar\'e map that small perturbations of $f$ will still admit a closed trajectory. Here we have another illustration of Remark \ref{nonrob}.
}
\end{example}

\medskip

\subsection{Partial symmetries}
Partial symmetries of differential equation systems were introduced in \cite{CGPart}, as a generalization of notions such as conditional symmetry and weak symmetry. We will briefly (and in a simplified manner) review the concept for first order ODEs, and discuss the connection to algebraic side conditions. As in \cite{CGPart} we focus on a local one-parameter transformation group $G(s,\,y)$ (in particular $G(0,\,y)=y$) induced by a smooth vector field $X_g$ on $U$. For our purpose it is convenient to slightly adjust the wording in the definition:
\begin{definition} (a) We say that $g$ is an {\em infinitesimal partial symmetry} of \eqref{ode} if there exists a solution $z(t)$ of $\dot x=f(x)$ such that $G(s,z(t))$ is also a solution for all $s$ near $0$.\\
 (b) We say that $g$ is an {\em infinitesimal partial orbital symmetry} of \eqref{ode} if there is a solution $z(t)$ of $\dot x=f(x)$ such that $t\mapsto G(s,z(t))$ parameterizes a solution orbit of \eqref{ode} for all $s$ near $0$.

\end{definition}
We recall the adjoint representation
\[
{\rm ad}\,g\,(f):=\left[g,\,f\right]
\]
and the formula
\begin{equation}\label{brackie}
\frac{\partial}{\partial s}D_2G(s,x)^{-1}\,f(G(s,x))= {\rm ad}\,g(f)\,(G(s,x))
\end{equation}
where $D_2$, as above, denotes the partial derivative with respect to the second variable. (See e.g. Olver \cite{Olv}, Prop.~1.64.) The next result (essentially taken from \cite{CGPart}, Prop.~1) relates partial symmetries to side conditions.
\begin{proposition}\label{parsymchar}{\em (a)} The smooth vector field $g$ is a partial symmetry of \eqref{ode} only if the sets
\[
W_k:=\left\{x\in U;\,\left({\rm ad}\,g\right)^k(f)\,(x)=0\right\},\quad k\geq 1
\]
have nonempty intersection.\\
{\em (b)} The smooth vector field $g$ is a partial orbital symmetry of \eqref{ode} only if the sets
\[
\widetilde W_k:=\left\{x\in U;\,\dim_{\mathbb K^n}\left<f(x),\,\left({\rm ad}\,g\right)^k(f)\,(x)\right>\leq 1\right\},\quad k\geq 1
\]
have nonempty intersection.
\end{proposition}
\begin{proof}(a) Let $G$ denote the flow of $g$, and let $z(t)$ be a solution of \eqref{ode} such that $G(s,\,z(t))$ is also a solution for all $s$ near $0$. Then 
\[
f(z(t))=D_2G(s,z(t))^{-1}\,f(G(s,z(t)))
\]
holds for all $t$ and $s$ near $0$; and differentiation with respect to $s$ yields, by \eqref{brackie} and an obvious induction,
\begin{equation}\label{fullbrackie}
\begin{array}{rrl}
0&=&\frac{\partial^k}{\partial s^k}D_2G(s,z(t))^{-1}\,f(G(s,z(t)))\\
  &=& D_2G(s,z(t))^{-1}{\rm ad}\,g^k(f)(G(s,z(t)).
\end{array}
\end{equation}
The assertion follows. \\
The proof of part (b) involves a reparameterization of time; thus the argument starts from
\[
\dot z(t)=\mu(s,t)\,f(z(t))
\]
with smooth $\mu$ and $\mu(0,t)=1$, but then works analogously.
\end{proof}
\begin{corollary}{\em (a)} In the analytic setting the vector field $g$ is a partial symmetry of $\dot x=f(x)$  if and only if the intersection of the $W_k$, $k\geq 1$ contains a nonempty invariant set of this equation. \\
{\em (b)} In the analytic setting the vector field $g$ is a partial orbital symmetry of $\dot x=f(x)$  if and only if there is a nonempty invariant set  of this equation which is contained in the intersection of the $\widetilde W_k$, $k\geq 1$. 
\end{corollary}
\begin{proof} For analytic $f$ and $g$, equation \eqref{fullbrackie} implies the Lie series formula
\begin{equation}\label{liebrackser}
D_2G(s,x)^{-1}\,f(G(s,x))= \sum_{k\geq 0}\frac{s^k}{k!}\left({\rm ad}\,g\right)^k(f)\,(x),
\end{equation}
from which in turn the assertions follow.
\end{proof}
\begin{remark}{\em
In any case, the existence of partial symmetries for $f$ implies the existence of particular side conditions. The simplest of these define $W_1$ resp. $\widetilde W_1$, and are explicitly given by
\begin{equation}\label{parsymside}
\left[g,\,f\right]=0,\quad\text{resp.}\quad \widetilde\Delta\left(\left[g,\,f\right],\,f\right)=0\text{ for every alternating $2$-form.}
\end{equation}
 Note the contrast to the symmetry case, where one has $g=0$ resp. $\widetilde\Delta(g,\,f)=0$ as simplest possible side conditions.}\\
\rightline {$\diamond$}
\end{remark}
\begin{example}\label{ex9}{\em Let
\[
g(x):=\left(\begin{array}{r} x_1\\
                                      -x_2\end{array}\right),\quad 
f(x):=\left(\begin{array}{c} x_1+ x_2+x_1^2x_2\\
                                      x_2 + x_1x_2^2\end{array}\right),\quad [g,\,f](x)=\left(\begin{array}{c} -2x_2\\
                                      0\end{array}\right)
\]
Since the zero set of the Lie bracket (given by $\psi:=x_2=0$) is indeed invariant for $f$, we have a partial symmetry $g$ as well as the side condition $\psi$ admitted by \eqref{ode}.
}
\end{example}
\subsection{Side conditions in dimension two}
In view of Propositions \ref{groupinv}, \ref{sigmainv} and \ref{parsymchar} we discuss side conditions for two-dimensional vector fields, with an obvious focus on invariant sets of dimension one. Here one obtains a rather clear picture relating nontrivial side conditions to nontrivial Lie bracket conditions.
The following facts about two-dimensional vector fields will be useful (see e.g. \cite{WPre}, Prop.~1.1 for a proof).
\begin{lemma}\label{twodimred}
Let $f$ and $g$ be smooth vector fields on the open set
$U\subseteq \mathbb K^2$, and assume that 
\[
\theta(x):=\det(f(x),g(x))\not=0 \mbox{ for }x\in\widetilde U\subseteq U,
\]
and $\widetilde U\not=\emptyset$. Then the identity
\begin{equation}\label{twodiminv}
\left[g,f\right]=\alpha f + \beta g
\end{equation}
holds on $\widetilde U$ with
\begin{equation}\label{twodcoeffs}
\alpha =\left(\frac{X_g(\theta)}{\theta}-{\rm div}\, g\right),\quad \beta=-\left(\frac{X_f(\theta)}{\theta}-{\rm div}\, f\right).
\end{equation}
\end{lemma}
One should not expect \eqref{twodiminv} to hold with smooth $\alpha$ and $\beta$ at any point where $\theta=0$. Actually, such an extension of $\beta$ beyond $\widetilde U$ is possible (roughly speaking)  if and only if the zero set of $\theta$ contains particular invariant sets for $f$. We will prove a precise version of this statement only for complex polynomial vector fields, to keep technicalities to a minimum. (See e.g.~Kunz \cite{Kun} for some notions of elementary algebraic geometry we will use below.)
\begin{proposition}\label{twodimside} Let $f$ and $g$ be polynomial vector fields on $\mathbb C^2$, with notation as in Lemma \ref{twodimred} and $\alpha$, $\beta$ from \eqref{twodcoeffs} (in particular these functions are rational). Let
\[
\det(f,\,g)=\theta=\sigma_1^{m_1}\cdots \sigma_r^{m_r}
\]
be the prime factorization, with pairwise relatively prime $\sigma_i$, $1\leq i\leq r$, and denote the zero set of $\sigma_i$ by $Y_i$. Then $\beta$ is regular at some point of $Y_j\setminus \bigcup_{i\not= j}Y_i $, $j\in\{1,\ldots,r\}$ if and only if $Y_j$ is invariant for \eqref{ode}.

\end{proposition}
\begin{proof} The zero set of $\theta$ is the union of the zero sets of the $\sigma_i$, all of which are non-empty due to the Hilbert {\em Nullstellensatz}. Also by virtue of the {\em Nullstellensatz}, $Y_j$ will be invariant if and only if $\sigma_j$ divides $X_f(\sigma_j)$ (see e.g. \cite{WPre} for a proof). This proves one direction of the equivalence. For the reverse direction assume that $\beta$ is regular at some $z\in Y_j\setminus \bigcup_{i\not= j}Y_i $ and use 
\[
X_f(\theta)/\theta=\sum m_i\,X_f(\sigma_i)/\sigma_i
\]
to see that $X_f(\sigma_j)/\sigma_j$ must be regular in $z$. This forces $X_f(\sigma_j)/\sigma_j$ to be polynomial. 

\end{proof}
\begin{corollary} \label{twodimcor} Let the situation and notation of Proposition \ref{twodimside} be given. Then the following are equivalent:\\
{\em (i)} The vector fields $f$ and $g$ are in involution on $U$; i.e., identity \eqref{twodiminv} holds with polynomial functions $\alpha$ and $\beta$ on $U$.\\
{\em(ii)} The zero set of $\theta$ is invariant for both $f$ and $g$.
\end{corollary}
\begin{remark}{\em
On the other hand, one may obtain every algebraic invariant set of a polynomial equation from ``partial involution" with some polynomial vector field, in the following sense. Let a polynomial system \eqref{ode} be given on $\mathbb C^2$ and let $\sigma$ be a polynomial such that its zero set $Y$ is invariant, but $\sigma$ is not a first integral of $f$.  (Thus $\sigma$ is a proper conditional invariant, or semi-invariant, of \eqref{ode}.) Choose the Hamiltonian vector field
\[
g:= h_\sigma=\left(\begin{array}{r}-\partial\sigma/\partial x_2\\
                                                      \partial\sigma/\partial x_1\end{array}\right),
\]
then the function $\beta$ in relation \eqref{twodiminv} is regular on a Zariski-open subset of $Y$. \\
To see this, recall that there is a nonzero polynomial $\lambda$ such that $X_f(\sigma)=\lambda\sigma$, due to invariance and the {\em Nullstellensatz}. By construction 
\[
\theta=\det(f,\,g)=X_f(\sigma)=\lambda \sigma\,;\quad X_f(\theta)/\theta = X_f(\lambda)/\lambda+\lambda.
\]
\rightline {$\diamond$}
}

\end{remark}

The results above can be easily transferred to the local analytic setting, with analogous proofs. Further extension to the global analytic case (on an open and connected set $U$) requires a restatement in weaker form, since the Noetherian and unique factorization properties will  be lost. For the smooth case one essentially obtains results about invariant local submanifolds, but one may have to deal with degenerate cases such as $\theta$ being identically zero on some open subset. 

\section{An application to reaction equations}
In this final section we show that side conditions appear naturally in the context of some applied problems, viz., for quasi-steady state (QSS) in chemistry and biochemistry. Side conditions are a mathematical incarnation of quasi-steady state assumptions for chemical species, and provide a computational approach to the detection of parameter regions where QSS phenomena arise.
\subsection{Background and motivation}
For some chemical reaction equations, in particular in biochemistry, one is interested in conditions that cause certain components of a solution to change slowly when compared to the overall rate of change. One speaks of quasi-steady state (resp.~a quasi-steady state assumption) in this case; see Atkins and de Paula \cite{ap}, p.~812 ff. on the chemistry background. Typically, the differential equation modelling the reaction depends on parameters (such as initial concentrations and rate constants), and one asks for conditions on these parameters which guarantee quasi-steady state. To illustrate the concept we consider the Michaelis-Menten system (for details see Segel and Slemrod \cite{SSl}).  In this fundamental model for an enzyme-catalyzed reaction, enzyme (E) and substrate (S) combine reversibly to a complex (C) which in turn degrades to enzyme and product (P). Symbolically we have
\[
E+S\rightleftharpoons C\rightharpoonup E+P.
\]
Denoting the concentrations by the corresponding lower-case letters, mass action kinetics and stoichiometry lead to the system
\begin{equation}\label{mimeirr}
\begin{array}{clccl}
\dot{s}&=-&k_1e_0s&+&(k_1s+k_{-1})c\\
\dot{c}&= &k_1e_0s&-&(k_1s+k_{-1}+k_2)c\\
\end{array}
\end{equation}
with relevant initial values $s(0)=s_0>0,\,c(0)=0$, and nonnegative rate constants $k_1$, $k_{-1}$ and $k_2$. Particular interest lies in QSS for the complex concentration $c$. The standard translation of QSS to mathematical terms works via interpretation as a singular perturbation problem; see Segel and Slemrod \cite{SSl} for a thorough discussion, and also the overview in \cite{gw}, subsection 8.2. In the present paper we will pursue a more general and at the same time more straightforward approach. This essentially goes back to Heinrich and Schauer \cite{hs-inv}, and is based on a different aspect, viz.~the existence of certain invariant sets. \\
For motivation, note that QSS for $c$ should imply 
\[
\left(\dot c=\right)\,k_1e_0s-(k_1s+k_{-1}+k_2)c\approx 0.
\] 
In practice, the stronger assumption 
\[
\phi(s,\,c):=k_1e_0s-(k_1s+k_{-1}+k_2)c= 0
\]
is used to express $c$ as a function of $s$, obtaining (upon substitution in the first equation) a one-dimensional differential equation for $s$. From a mathematical perspective this procedure is consistent only if the zero set of $\phi$ is actually invariant for \eqref{mimeirr}.
Heinrich and Schauer \cite{hs-inv} relaxed the invariance requirement by stipulating that the actual trajectory remain close to the zero set of $\phi$ (for the time period of interest). The Heinrich-Schauer condition (which was put in broader context in \cite{nw09}) involves rather intricate estimates and therefore is cumbersome to implement, but the sharper invariance requirement naturally leads to LaSalle type side conditions, and a computationally feasible approach.
Thus we augment the reasoning in \cite{hs-inv} with the following observation: Assume that for some parameters
\[
p^*:=(e_0^*,\,k_1^*,\,k_{-1}^*,\,k_2^*)\in \mathbb R^4_+
\]
the zero set of $\phi$ is actually invariant for system \eqref{mimeirr}. Then it is reasonable to suspect that the Heinrich-Schauer condition should be satisfied for small perturbations of this parameter set, and we will verify this in the next subsection.

\subsection{Side conditions for parameter-dependent systems}
First we need to specify the somewhat loose statement just given.
Thus we present and prove a general formulation which is applicable beyond the QSS scenario. Consider a parameter-dependent system
\begin{equation}\label{odep}
\dot x = f(x;p),\quad (x,p)\in \mathbb R^n\times \mathbb R^m
\end{equation}
with the right-hand side smooth on an open subset $D$ of $\mathbb R^n\times \mathbb R^m$. (Less restrictive asumptions would suffice for the purpose of this subsection.) We denote the local flow by $F(t,y;\,p)$.
Moreover consider smooth functions
\begin{equation}
\phi_1,\,\ldots,\phi_r:\, D\to \mathbb R.
\end{equation}
In practice, these functions may come from chemical intuition, or from educated guesses (such as QSS assumptions for certain chemical species), and the common zero set of these functions is conjectured to be close to an invariant set. The following proposition yields criteria to verify such a conjecture.
\begin{proposition}\label{asnear}
 Let $p^*\in\mathbb R^m$ such that the equations
\[
\phi_1(x,\,p^*)=\cdots =\phi_r(x,\,p^*)=0
\]
define a local $s$-dimensional submanifold $Y_{p^*}$ of $\mathbb R^n$ which is invariant for the system \eqref{odep}. Moreover let $y^*\in Y_{p^*}$ and assume that the Jacobian for suitable $n-s$ functions among $\phi_1,\ldots,\phi_r$ has rank $n-s$ at $y^*$. Then the following hold.\\
{\em (a)} There is a compact neighborhood $K$ of $y^*$ and a neighborhood $V$ of $p^*$ such that $Y_{p^*}\cap K$ is compact, and that for every $p\in V$ the set defined
by the equations
\[
\phi_1(x,\,p)=\cdots =\phi_r(x,\,p)=0
\]
contains an $s$-dimensional local submanifold $Y_p$ which has nonempty compact intersection with $K$. Furthermore, for every $\varepsilon>0$ there is a $\delta>0$ such that
\[
{\rm dist}\,(x,\,Y_{p^*}):=\inf_{z\in Y_{p^*}}\Vert x-z\Vert<\varepsilon \text{ for all } x\in Y_p \text{  whenever  }  \Vert p-p^*\Vert<\delta.
\]
{\em (b)} Let $T>0$ such that $F(t,y;p^*)$ exists on the interval $[0,\,T]$ for all $y\in Y_{p^*}\cap K$. Then for each $\rho>0$ there exists $\theta>0$ with the following property: For every $p\in V$ with $\Vert p-p^*\Vert<\theta$ and every $z\in Y_p$ the solution $F(t,z;p)$ exists on the interval $[0,\,T]$, and 
\[
{\rm dist}\,(F(t,z;p),\,Y_{p})<\rho \text{ for all } t\in [0,T].
\]
{\em (c)} Given $p$ sufficiently close to $p^*$, assume (with no loss of generality) that $x_{1},\ldots, x_{{s}}$ are local coordinates on $Y_p$, and 
\[
x_k=\eta_k(x_1,\ldots,x_{s};\,p),\quad s+1\leq k\leq n, \text{     on  }Y_p\cap K.
\]
Then the solution of 
\[
\dot x_i=f_i(x_1,\ldots,x_{s},\eta_{s+1},\ldots,\eta_{n};\,p),\quad 1\leq i\leq s,
\]
combined with $x_k=\eta_k(x_1,\ldots,x_{s};\,p)$ for $k>s$, converges on $[0,\,T]$ to the solution of \eqref{odep} as $p\to p^*$.
\end{proposition}
\begin{proof} Part (a) is a consequence of the implicit function theorem and a compactness argument, while parts (b) and (c) follow from (a) and standard dependence theorems. 

\end{proof}
\begin{remark}{\em 
(a)  A more comprehensive generalization of Heinrich and Schauer's concept \cite{hs-inv}, called {\em near-invariance}, was introduced and discussed in \cite{nw09}. One consequence of Proposition \ref{asnear} is that for every $\sigma>0$ there exists $\eta>0$ such that $Y_p\cap K$ is $\sigma$-nearly invariant whenever $\Vert p-p^*\Vert<\eta$. (The stronger property that $\sigma$ may be chosen arbitrarily small is not required in the more general notion from \cite{nw09}.) As shown by the examples in \cite{nw09}, finding (sharp) estimates for near-invariance may be quite involved.\\
(b) One may encounter the degenerate case that $Y_{p^*}$ consists of stationary points only. Then the statement of Proposition \ref{asnear} is correct but not particularly strong. On the other hand, this degenerate scenario is actually one prerequisite for application of the classical singular perturbation results by Tikhonov \cite{tikh} and Fenichel \cite{fenichel}; see \cite{gw}, Thm.~8.1.
If the additional hypotheses for Tikhonov's theorem are fulfilled then one obtains a sharper result (on the slow time scale) in lieu of the proposition above. There exists a more systematic (and more intricate) approach to finding ``small parameters'' for singular perturbation scenarios (see the recent dissertation \cite{godiss} by A.~Goeke and also \cite{GWZ}), but side conditions still provide an easy-to-use tool for detection.
}
\rightline {$\diamond$}
\end{remark}
\subsection{Some applications}
We consider two famous reaction equations which have been extensively discussed in the literature. 
\subsubsection{Michaelis-Menten}
The Michaelis-Menten system is probably the most famous among the systems exhibiting QSS. We will abbreviate \eqref{mimeirr} as $\dot x =f(x,\,p)$, with $x=(s,c)$. Three types of QSS assumption have been discussed in the literature:
\begin{itemize}
\item QSS for complex: $\phi=\psi_1:= X_f(c)=k_1e_0s-(k_1s+k_{-1}+k_2)c$ (also known as {\em standard QSS}).
\item QSS for substrate: $\phi= \psi_2:=X_f(s)=-k_1e_0s+(k_1s+k_{-1})c$. This is also known as {\em reverse QSS}; see Segel and Slemrod \cite{SSl}.
\item QSS for total substrate: $\phi=\psi_3:= X_f(s+c)=- k_2c$. This is also known as {\em total QSS}; see Borghans et al. \cite{bbs}.
\end{itemize}
We determine parameter combinations which yield invariance, and thus allow the application of Proposition \ref{asnear}.
\begin{proposition} Consider the Michaelis-Menten system with nonnegative parameters $e_0$ and $k_i$. Then:
\begin{itemize}
\item The submanifold defined by $c=0$ is invariant for system \eqref{mimeirr} if and only if $e_0=0$ or $k_1=0$.
\item The zero set of $\psi_1$ contains a one-dimensional invariant submanifold of system \eqref{mimeirr} if and only if any one of the following holds:\\
(i) $e_0=0$; (ii) $k_1=0$;  (iii) $k_2=0$.\\
In the first two cases, the invariant manifold is given by $c=0$; in the third case it has the representation $c=k_1e_0s/(k_1s+k_{-1})$ (in particular $c=e_0$ if $k_{-1}=0$). In all cases the invariant manifold consists of stationary points only.
\item The zero set of $\psi_2$ contains a one-dimensional invariant submanifold of system \eqref{mimeirr} which is not among those previously discussed if and only if $k_{-1}=0$. In this case the manifold is given by $s=0$.
\item All the one-dimensional invariant manifolds contained in the zero set of $\psi_3$ are among those of the standard QSS case.
\end{itemize}
\end{proposition}
\begin{proof} Since 
\[
X_f(c)=k_1e_0s-\left(\cdots\right)\cdot c
\]
(with $(\cdots)$ standing for some polynomial whose explicit form is of no relevance here), the common zero set of $c$ and $X_f(c)$ contains just the point $0$ whenver $k_1e_0\not=0$. On the other hand, $k_1e_0=0$ implies invariance by Remark \ref{invcrit}. This proves the first assertion. For the following we note that $X_f(s)+X_f(c)=-k_2c$.\\
A straightforward computation shows
\[
\begin{array}{rcl}
X_f(\psi_1)=X_f^2(c)&=&k_1(e_0-c)X_f(s)-(k_1s+k_{-1}+k_2)X_f(c)\\
                                                  &=& -k_1(e_0-c)\cdot k_2c+\left(\cdots\right)\cdot\psi_1.
\end{array}
\]
Thus the zero set $Y$ of $\psi_1$ is invariant, by Remark \ref{invcrit}, in case $k_1=0$ or $k_2=0$. If $k_1k_2\not=0$ then $Y$ must either contain the zero set of $c$ as an invariant set (which was discussed above), or the zero set of $e_0-c$. Since
\[
X_f(e_0-c)=\left(\cdots\right)\cdot(e_0-c)+ (k_{-1}+k_2)\cdot c
\]
and $k_2>0$, common zeros of $e_0-c$ and $X_f(e_0-c)$ exist only when $c=0$.\\
A similar computation yields
\[
X_f(\psi_2)=\left(\cdots\right)\cdot \psi_2-(k_1s+k_{-1})\cdot k_2c.
\]
Assume that $k_1\not=0$. Then $k_{-1}=0$ implies $\psi_2=-k_1(e_0-c)s$, and from previous arguments it is known that invariance of the line $e_0-c=0$ implies $e_0=0$. Thus only the case $s=0$ yields a new invariant set. Moreover, the set defined by $k_1s+k_{-1}=0$ is invariant only when $k_{-1}=0$, in view of $X_f(s)=-k_1e_0 s+(k_1s+k_{-1})\cdot(\cdots)$.\\
The final assertion follows directly from previous arguments.
\end{proof}
By Proposition \ref{asnear}, for small $e_0$ one will have an invariant manifold close to $c=0$, for small $k_2$ one will have an invariant manifold close to the curve $c=k_1e_0s/(k_1s+k_{-1})$ , and so on.
We provide more details for two cases, with the results stated somewhat informally.
\begin{itemize}
\item For sufficiently small $e_0$ (the other parameters being fixed and $>0$), solutions starting close to the set defined by $\psi_1=0$, i.e.
\[
c=\frac{k_1e_0s}{k_1s+k_{-1}+k_2}
\]
will remain close for an extended duration of time, and the solution will be close to a solution of the familiar reduced equation
\[
\dot s = \frac{-k_1k_2e_0 s}{k_1s+k_{-1}+k_2}.
\]
An analysis via singular perturbation theory yields the same reduced equation; see Segel and Slemrod \cite{SSl}. But one should emphasize that the two procedures generally lead to reduced systems which are different, and the difference is of the same order in the small parameter as the systems themselves. For Michaelis-Menten this phenomenon occurs for small parameter $k_2$.
\item  For sufficiently small $k_{-1}$ (the other parameters being fixed and $>0$), solutions starting in the set defined by $\psi_2=0$, i.e.
\[
s=\frac{k_{-1}c}{k_1(e_0-c)}
\]
will remain close to this set for an extended duration of time, and the solution will be close to a solution of the reduced equation
\[
\dot c=-k_2c.
\]
\end{itemize}
The second scenario does not represent a standard singular perturbation problem with small parameter $k_{-1}$, since the zero set of $\psi_2$ contains non-stationary points when $k_{-1}=0$. Thus  the method outlined in Proposition \ref{asnear} also yields (asymptotic) invariant sets that one cannot trace back to singular perturbation phenomena.
\subsubsection{Lindemann-Hinsley}
The Lindemann-Hinsley system
\begin{equation}\label{lihi}
\begin{array}{rcl}
\dot a&=&-k_1a^2+k_{-1}ab\\
\dot b&=&k_1a^2-k_{-1}ab-k_2b
\end{array}
\end{equation}
models a two-stage degradation process of a chemical species A, with activated stage B. More background and a phase plane analysis are given in Calder and Siegel \cite{cs}. Again the right-hand side will be abbreviated by $f(x,p)$, with obvious variables and parameters. We are interested in QSS for the concentration $b$ of activated molecules, thus we have
\[
\phi=X_f(b)=k_1a^2-k_{-1}ab-k_2b.
\]
\begin{proposition} The zero set of $\phi$ contains a one-dimensional invariant submanifold of system \eqref{lihi} (with nonnegative parameters) if and only if (i) $k_1=0$ or (ii) $k_2=0$. In the first case the invariant set is given by $b=0$. In the second  case there exist two invariant manifolds, given by $a=0$, resp. by $k_1a-k_{-1}b=0$. In any case the invariant sets are made up of stationary points only.

\end{proposition}
\begin{proof} One finds
\[
X_f(\phi)=X_f^2(b)=-k_2b\cdot\left(2k_1a-k_{-1}b\right)+\left(\cdots\right)\cdot\phi.
\]
If $k_2=0$ then the remaining assertions are immediate. If $k_2\not=0$ then the zero set of $\phi$ must either contain the zero set of $b$, which forces $k_1=0$, or the set given by $2k_1a-k_{-1}b=0$. The latter leads to the contradiction $k_2=0$.
\end{proof}
By Proposition \ref{asnear} we see, for instance, that for $k_2\to 0$ (and the other parameters constants $>0$) any solution starting close to the line given by $k_1a-k_{-1}b=0$ will remain close for an extended duration of time, and the solution of \eqref{lihi} is approximated by the reduced equation
\[
\dot b = -k_2b.
\]
A singular perturbation analysis yields the same reduced equation with a stronger justification; see Calder and Siegel \cite{cs}, and Goeke \cite{godiss}.
\begin{remark}{\em 
The main purpose of this final section was to present a natural application of side conditions in a different -- and perhaps unexpected -- field, and to show by (simple but relevant) examples that the side condition approach provides a conceptually straightforward and computationally feasible way to determine QSS conditions for prescribed variables. Moreover, the usual types of reaction equations (polynomial, due to mass action kinetics) are accessible by methods of algorithmic algebra. This will be the subject of forthcoming work.}\\
\rightline {$\diamond$}
\end{remark}

\end{document}